\documentclass[11pt,reqno]{amsart}
\usepackage{amsmath,amsthm,amsxtra,amssymb,latexsym,mathrsfs,mathtools}
\usepackage[a4paper,top=2.5cm,bottom=2cm,left=2.5cm,right=2.5cm,marginparwidth=1.75cm]{geometry}
\usepackage{verbatim}
\usepackage{amscd}
\usepackage[all]{xy}
\usepackage{eucal}
\usepackage{hyperref}
\usepackage{color}
\usepackage{float}
\usepackage{tikz,tikz-cd}
\usepackage{enumitem}

\newtheoremstyle{plain-alt}
  {}
  {}
  {\itshape}
  {}
  {}
  {. }
  {0em}
  {\thmnumber{#2}.\,\,\thmname{{\it #1}}\thmnote{ (#3)}}
\newtheoremstyle{definition-alt}
  {}
  {}
  {\normalfont}
  {}
  {}
  {. }
  {0em}
  {\thmnumber{#2}.\,\,\thmname{{\it #1}}\thmnote{ (#3)}}
\theoremstyle{plain-alt}

\theoremstyle{plain}
\swapnumbers
\newtheorem{lemma}{Lemma}[section]

\newtheorem{corollary}[lemma]{Corollary}
\newtheorem{theorem}[lemma]{Theorem}

\theoremstyle{definition}
\newtheorem{definition}[lemma]{Definition}

\newtheorem{remark}[lemma]{Remark}

\newtheorem{example}[lemma]{Example}

\newcommand{\Sol}{\operatorname{Sol}}

\newcommand{\trop}{\mathcal T}

\newcommand{\lin}{\operatorname{Lin}}

\title{Infinite matroids in tropical differential algebra}
\author[F. Aroca et al.]{F. Aroca, L. Bossinger, S. Falkensteiner, C. Garay Lopez, L. R. Gonzalez-Ramirez, C. V. Valencia Negrete}

\address{
Instituto de Matem\'aticas Unidad Oaxaca, Universidad Nacional Aut\'onoma de M\'exico,Le\'on 2, altos, 
Centro Hist\'orico,
68000 Oaxaca,
M\'exico
}
\email{fuen@im.unam.mx}

\address{
Instituto de Matem\'aticas Unidad Oaxaca, 
Universidad Nacional Aut\'onoma de M\'exico,
Le\'on 2, altos, 
Centro Hist\'orico,
68000 Oaxaca,
Mexico}
\email{lara@im.unam.mx}

\address{Max Planck Institute for Mathematics in the Sciences Leipzig, 
Inselstrasse 22,
04103 Leipzig,
Germany}
\email{seb.falkensteiner@gmail.com}

\address{Centro de Investigaci\'on en Matem\'aticas, A.C. (CIMAT). Jalisco S/N, Col. Valenciana CP. 36023 Guanajuato, Gto, M\'exico}
\email{cristhian.garay@cimat.mx}

\address{Instituto Polit\'ecnico Nacional, Escuela Superior de F\'isica y Matem\'aticas, Unidad Profesional Adolfo L\'opez Mateos Edificio 9, 07738, Cd. de M\'exico, M\'exico}
\email{lrgonzalezr@ipn.mx}

\address{Universidad Iberoamericana, A. C. (UIA).
Prolongaci\'{o}n Paseo de la Reforma 880, 
Lomas de Santa Fe,
01219 Ciudad de M\'{e}xico, M\'exico}
\email{carla.valencia@ibero.mx}

\date{\today}

\begin{document}




\maketitle

\bigskip

\begin{abstract}
We consider a finite-dimensional vector space $W\subset K^E$ over an arbitrary field $K$ and an arbitrary set $E$. We show that the set $\mathcal{C}(W)\subset 2^E$ consisting of the minimal supports of $W$ are the circuits of a matroid on $E$.
In particular, we show that this matroid is cofinitary (hence, tame). 
When the cardinality of $K$ is large enough (with respect to the cardinality of $E$), then the set $\trop(W)\subset 2^E$ consisting of all the supports of $W$ is a matroid itself.

Afterwards we apply these results to tropical differential algebraic geometry and study the set of supports $\trop(Sol(\Sigma))\subset (2^{\mathbb{N}^{m}})^n$ of spaces of formal power series solutions $\text{Sol}(\Sigma)$ of systems of linear differential equations $\Sigma$  in differential variables $x_1,\ldots,x_n$ having coefficients in the ring ${K}[\![t_1,\ldots,t_m]\!]$. 
If $\Sigma$ is of differential type zero, then the set $\mathcal{C}(Sol(\Sigma))\subset (2^{\mathbb{N}^{m}})^n$ of minimal supports defines a matroid on $E=[n]\times\mathbb{N}^{m}$, and if the cardinality of $K$ is large enough, then the set of supports $\phi\circ\trop(Sol(\Sigma))$ itself is a matroid on $E$ as well. 
By applying the fundamental theorem of tropical differential algebraic geometry (\textsc{fttdag}), we give a necessary condition under which {the set of solutions $Sol(U)$ of a  system $U$ of tropical linear differential equations to be a matroid.}

We also give a counterexample to the \textsc{fttdag} for systems $\Sigma$ of linear differential equations over countable fields. In this case, the set $\phi\circ\trop(Sol(\Sigma))$ may not form a matroid.
\end{abstract}

\section{Introduction}
A fundamental concept in tropical algebraic geometry is the {\it tropicalization} $trop(X,v)\subseteq (\mathbb{R}\cup\{-\infty\})^n$ of an algebraic variety $X\subset K^n$ defined over a valued field $K=(K,v)$. 
Assume $X$ is defined by an ideal $I$ with constant coefficients (\emph{i.e.} $I\subset k[x_1,\dots,x_n]$ where $k\subset K$ is a field with trivial valuation $v_0:k\xrightarrow[]{}\{-\infty,0\}$). 
In this case we may consider the $K$-points of $X$ and tropicalize with respect to $v$ which yields a polyhedral fan $trop(X,v)$.
A classical result from tropical algebraic geometry states that $trop(X,v)$ coincides with the Bergman fan $\mathcal B(X)$ \cite[Theorem 9.6]{Sturmfels_solving}. If $X$ is a linear space, the Bergman fan can be obtained from its matroid $M(X)=(\{1,\ldots,n\},\mathcal{C}(X))$ (see e.g. \cite[p.165]{MS_tropicalBook} or \cite{Ardila-Klivans}).
On the other hand, we may consider the trivial valuation $v=v_0$. Note that the set 
\begin{equation}
    \label{eq:classical_trop}
    v_0(X):=\{(v_0(p_1),\ldots,v_0(p_n))\in \{-\infty,0\}^n\::\:(p_1,\ldots,p_n)\in X\},
\end{equation}
consists of all the {\it  supports} of the points of $X$: it suffices to consider the identification $\{-\infty,0\}^n\cong 2^{[n]}$ as the set of all the indicator functions of the subsets of $[n]$. 
In general we have $v_0(X)\subseteq trop(X,v_0)$, and this inclusion may be proper; for example, if $X\subset \mathbb{F}_2^3$ is the linear space spanned by  $\{(0,1,1), (1,0,1)\}$, then $$v_0(X)\cup\{(0,0,0)\}=trop(X,v_0),$$
which says that~\eqref{eq:classical_trop} carries potentially more information than the matroid $M(X)$.
We show that, if $X\subset K^E$ is a finite-dimensional vector space, then the above situation cannot happen if the cardinality of $K$ is large enough with respect to the cardinality of the set $E$; that is, the set of supports $v_0(X)\subset 2^E\cong\{0,-\infty\}^E$ and the matroid $M(X)=(E,\mathcal{C}(X))$  associated to $X$  can be identified with each other via
\begin{equation}
\label{eq:correspondence}
    \begin{tikzcd}
    (E,\mathcal{C}(X))=M(X) \arrow[rr, "\text{scrawls}", shift left=3] &  & v_0(X)=(E,\mathcal{S}(X)) \arrow[ll, "\text{circuits}", shift left=3].
    \end{tikzcd}
\end{equation}
To do this, we use the notion of infinite matroids under the cryptomorphisms of circuits~\cite{Infinte_matroids} and scrawls~\cite{BowlerCarmesin:2018}.
Up to our knowledge, this result is new in both cases when $E$ is finite and when $E$ is infinite. 
While it is possible to obtain that $M(X)=(E,\mathcal{C}(X))$ is a matroid using the results of thin sums families and their associated thin sums systems from \cite{AB2015}, we offer an independent proof which has the advantage of being more elementary. We believe that both proofs have their own advantages, so we have decided to include both.

Our interest in this type of questions comes from the theory of tropical differential algebraic geometry, where the tropicalization of the set of formal power series solutions of systems of homogeneous linear differential equations appears as a set of the form $v_0(X)\subset 2^E$ where $E$ is infinite, as in \eqref{eq:trop:eqdiff:chngd}. This is why in this paper we deal mainly with duals of {\it representable matroids} over an infinite set of labels $E$ (see Theorem~\ref{thm-tame}), which are not representable in the usual sense (see Remark \ref{rem:rep}), but we show that under the above hypothesis they do arise as the semigroup of the set of  supports of a vector space.

Also, we aim to further study the Boolean formal power series solutions of systems of tropical homogeneous linear differential equations. This study was initiated in~\cite{Trop_diff_eq} and it is a natural continuation for the tropical aspects of the differential algebraic theory of such systems, following the development of classical tropical geometry.
\subsection{Analogies between classical and tropical differential algebraic geometry}
The theory of tropical differential algebraic geometry was initiated by Grigoriev~\cite{Trop_diff_eq}, the fundamental theorem was proved by Aroca, Garay and Toghani in~\cite{Fundamental_theorem_TropDiffAlgGeom}, and is becoming an established active field of research with several contributions such as~\cite{MeretaGian,Fink-Toghani,Mereta_fundThm,chinos}.

Since the beginning, a natural question arose: which concepts of classical tropical algebraic geometry over a valued field $(K,v)$ can be generalized to the differential setting? More specifically, can tropical differential algebraic geometry be regarded as an infinite version of classical tropical algebraic geometry?

A cornerstone of classical tropical algebraic geometry is the fundamental theorem of tropical algebraic geometry (\textsc{fttag},~\cite[Theorem 3.2.3]{MS_tropicalBook}, a generalization of Kapranov's Theorem \cite{MR2289207} for hypersurfaces), which describes in three ways the tropicalization $trop(X,v)\subset \mathbb{R}^n$ coming from algebraic subvarieties $X\subset (K^*)^n$, where $K$ is an algebraically closed field and $v:K^*\xrightarrow[]{}\mathbb{R}$ is a non-trivial valuation. 
A tropical (partial) differential analogue of this result (\textsc{fttdag}) was successfully constructed in~\cite{Fundamental_theorem_TropDiffAlgGeom,falkensteiner2020fundamental,falkensteiner2023initials} giving three different descriptions of the tropical space of formal Boolean formal power series $v_0(X)\subset \mathbb{B}[\![t_1,\ldots,t_m]\!]^n$ that come from differential algebraic (DA) varieties $X\subset {K}[\![t_1,\ldots,t_m]\!]^n$, where ${K}$ is an uncountable algebraically closed field of characteristic zero and $v_0$ is the trivial valuation. Note that in~\cite{boulier2021relationship} it is shown that the uncountability condition can be replaced by countably infinite transcendence degree over the field of definiton of $X$.

A natural source of finite-dimensional vector spaces $X\subset K^E$ with  $E$ infinite are the sets of solutions of systems of homogeneous linear differential equations of differential type zero with coefficients in $K[\![t_1,\ldots,t_m]\!]$. We explore its consequences for distinct fields $K$ satisfying or not the conditions of the \textsc{fttdag}. Our main results on matroids (Theorem~\ref{Teorema: es un matroide} and Theorem~\ref{Teorema Son los Scrawls}) lead to the following result (Theorem~\ref{Teorema_Matroide_Lin}):
\begin{theorem}
\label{Thm_4.2-1.1}
Fix two positive integers $m,n$. Let $E=\mathbb{N}^m$ and let $\trop : K^E\longrightarrow 2^E$ be the support map (Definition~\ref{def:supp_map}). Let $\Sigma\subset {K}_{m,n}$ be a system of homogeneous linear differential equations of differential type zero and let $\trop(\Sol(\Sigma))\subset (2^{E})^n$ be the set of supports of $\Sol(\Sigma)$. Then 
\begin{enumerate}
    \item the minimal elements $\mathcal{C}(\Sol(\Sigma))$ of $\trop(\Sol(\Sigma))$ define the circuits of a matroid on $[n]\times\mathbb{N}^{m}$;
    \item if the cardinality of $K$ is large enough, then $\trop(\Sol(\Sigma))$ itself is the set of scrawls of $\mathcal{C}(\Sol(\Sigma))$.
\end{enumerate}
\end{theorem}
Obtaining that $\trop(\Sol(\Sigma))$ is the set of scrawls of $\mathcal{C}(\Sol(\Sigma))$ (see item (1)), implies in particular that, if we have two supports of solutions of $\Sigma$, then their union appears also as the support of a solution of $\Sigma$; this is, $\trop(\Sol(\Sigma))$ is a semigroup (with respect to the operation of union of sets). 
In item (2), the assumption on the cardinality of $K$ is strict as we show in Section~\ref{sec:counterexp}.

The previous result is about the structure of the tropicalization of the set of formal solutions of a classical system. On the other hand, following D. Grigoriev~\cite{Trop_diff_eq}, one can study from an algebraic and combinatorial perspective the {set of solutions} $X= \Sol(U)\subset \mathbb{B}[\![t_1,\ldots,t_m]\!]^n$ associated to a system $U\subset\mathbb{B}_{m,n}$ of tropical linear differential equations, disregarding wether or not $U$ is realizable in some field $K$, this is, independent of the existence of linear systems $\Sigma\subset K_{m,n}$ such that $U=trop(\Sigma)$ as in Definition~\ref{def:trop_pol}. These {sets of solutions} are always semigroups even for the case of partial differential equations $m>0$, as we show in Theorem~\ref{prop:B-Module}.

\subsection{Statement of results}
In Theorem~\ref{Teorema: es un matroide} we show that if $\{0\}\neq W\subset K^E$ is a finite-dimensional $K$-vector subspace, then the pair $M(W)=(E,\mathcal{C}(W))$ is a matroid, where $\mathcal{C}(W)\subset 2^E$ is the set of vectors in $W$ having minimal (nonempty) support. 
This seems to be known for experts, but we provide a rigorous proof of this.
We also show that any element of $\trop (W)$ is a union of circuits, and in Theorem~\ref{Teorema Son los Scrawls} we show that if $\# E < \# S(K)$, where $S(K)=K \cup \{K\}$ denotes the set-successor of $K$,
then also the converse holds true. In Theorem~\ref{prop:B-Module} we show that the {the set of solutions} $\Sol(U)\subset \mathbb{B}[\![T]\!]^n$ associated to a system $U\subset \mathbb{B}_{m,n}$ of homogeneous linear tropical differential equations is a semigroup.
In particular, using the \textsc{fttdag}, we give in Corollary~\ref{cor:real} a necessary condition for {the set of solutions} $X= \Sol(U)\subset \mathbb{B}[\![t_1,\ldots,t_m]\!]^n$ to be a matroid (or a {matroid of scrawls}, as in Definition \ref{def_trop_semigp}). 
As a consequence of the previous results, in Section~\ref{sec:counterexp}, we give a counterexample for the \textsc{fttdag} in the case of linear differential equations over countable fields.

\subsection{Roadmap}
The paper is organized as follows. In Section~\ref{sec:matroid}, we introduce standard preliminary material on matroid theory and we prove Theorem~\ref{Teorema: es un matroide}. In Section~\ref{Section:Scrawls}, we study the matroid of scrawls and prove Theorem~\ref{Teorema Son los Scrawls}.
In Section~\ref{section:diff_tls}, we discuss the theory of algebraic differential equations with coefficients in the ring ${K}[\![t_1,\ldots,t_m]\!]$ over an arbitrary field $K$, and we recast the result of the previous two sections for the case of homogeneous systems of linear differential equations of differential type zero. In Section~\ref{Section:Connections}, we discuss  tropical differential equations, go further and analyze the special case in which $K$ satisfies the hypotheses of the \textsc{fttdag}.
\section{The infinite matroid induced by a finite dimensional subspace of a vector space}
\label{sec:matroid}
In this section we denote by $E\neq\emptyset$ an arbitrary set and by $2^E$ the power set of $E$, which is ordered by inclusion. We consider $2^E$ as a semigroup endowed with the set union as operation.

\subsection{Basic theory of infinite matroids}

A matroid on $E$ may be given in terms of different collections of subsets of $E$: the circuits, the independent sets, or the bases; these are only some possibilities among the many ways of defining matroids when $E$ is finite. 
The matroid axioms were shown to be equivalent by Whitney~\cite{Whitney:1935} in the finite case. 
For the infinite case, the Whitney axioms need to be completed, see~\cite{Infinte_matroids}. 
In~\cite{BowlerCarmesin:2018}, another possibility to define matroids is exhibited by using scrawls.

\begin{definition}\label{Def_CIM}
Let $\mathcal{C}\subset 2^E$. We call $\mathcal{C}$ the \textit{set of circuits} if it satisfies the following axioms:
\begin{enumerate}
\item $\emptyset \notin \mathcal{C}$; \label{El vacio no pertenece a los circuitos}
\item No element of $\mathcal{C}$ is a subset of another; \label{Un circuito no puede estar contenido en otro}
\item \label{sumacircuitos} Whenever $X\subset C\in\mathcal{C}$ and $\{C_x\::\:x\in X\}$ is a family of elements of $\mathcal{C}$ such that $x\in C_y$ iff $x=y$ for all $x,y\in X$, then for every $z\in C\setminus(\bigcup_xC_x)$ there is an element $C'\in \mathcal{C}$ such that $z\in C'\subset (C\cup \bigcup_xC_x)\setminus X$;
\item \label{HayMaximalDeIndependientes} The set of $\mathcal{C}$-independents $\mathcal{I}(\mathcal{C}):=\{I\subset E\::\: C\not\subset I\:\forall C\in \mathcal{C}\}$ satisfies 
\begin{equation*}
\text{if }I\subset X\subset E \text{ and }I\in\mathcal{I}, \text{ then }\{I'\in \mathcal{I}\::\:I\subset I'\subset X\}\text{ has a maximal element}.
\end{equation*}
\end{enumerate}
In this case, $M=(E, \mathcal{C})$ is a \textit{matroid}, and $\mathcal{C}$ is called the \textit{set of circuits of $M$}.
\end{definition}

\begin{definition}
Let $M=(E,\mathcal{C})$ be a matroid given in terms of its circuits. An \textit{independent set} is a subset of $E$ that contains no circuits (compare item~\ref{HayMaximalDeIndependientes} in Definition~\ref{Def_CIM}). A \textit{basis} is a maximal independent set.
\end{definition}

With abuse of notation, a matroid $M$ is also denoted by $M=(E,\mathcal{B})$, where the elements in $\mathcal{B}$ are the \textit{bases} of $M$.
We also have the following definition of matroid via sets of \textit{scrawls}, which are the unions of circuits \cite[Section 2.2]{BowlerCarmesin:2018}.

\begin{definition}
\label{def_trop_semigp}
Let $\mathcal{S}\subset 2^E$. We call $\mathcal{S}$ the \textit{set of scrawls} if it satisfies the following axioms:
\begin{enumerate}
    \item $\mathcal{S}$ is a semigroup, i.e. any union of elements in $\mathcal{S}$ is in $\mathcal{S}$;
    \item $\mathcal{S}$ satisfies the conditions (3) and (4) in Definition~\ref{Def_CIM}, where  $\mathcal{I}(\mathcal{S}):=\{I\subset E\::\: C\not\subset I\:\forall C\in \mathcal{S}\setminus\{\emptyset\}\}$.
\end{enumerate}
In this case, $M=(E, \mathcal{S})$ is a \textit{matroid} and $\mathcal{S}$ is called the \textit{set of scrawls of $M$}.
\end{definition}
Definitions \ref{Def_CIM} and \ref{def_trop_semigp} are equivalent, indeed, given a set of circuits, we get a set of scrawls by taking all the finite unions of them, conversely, given a set of scrawls, we get a set of circuits by taking the minimal nom-empty elements with respect to inclusion. Notice that a set of scrawls naturally carries the structure of a semigroup and simultaneously encodes the information of a matroid.

Let $M=(E,\mathcal{B})$ be a matroid with basis elements $\mathcal{B}$. 
By~\cite[Theorem 3.1]{Infinte_matroids}, the complements $\mathcal{B}^*:= \{E \setminus B \::\: B \in \mathcal{B} \}$ form another matroid $M^*=(E,\mathcal{B}^*)$, called the \textit{dual matroid} to $M$. 
The circuits of $M^*$ are called \textit{cocircuits} {of $M$}. 
A matroid $M$ is called \textit{tame} if every intersection of a circuit and a cocircuit is finite~\cite{Infinte_matroids}, otherwise $M$ is called {\it wild}.

If all the circuits of a matroid $M$ (respectively $M^*$) are finite
then $M$ is called finitary (respectively cofinitary). These matroids are tame~\cite{AB2015}.
The matroids considered in this paper are cofinitary as we will show in Theorem~\ref{thm-tame}.

\subsection{The set of supports of a finite-dimensional vector space}
In this section, we consider the vector space $K^E$, where $E$ is as above and  $K$ is any field. We will consider finite-dimensional $K$-vector subspaces $\{0\}\neq W\subset K^E$.
\begin{definition}
\label{def:supp_map}
The \textit{support map} $\trop : K^E\longrightarrow 2^E$ is the mapping $(a_i)_{i\in E}\mapsto \{ i\in E \::\: a_i\neq 0\}$. We call $\trop(v)$ the \textit{support} of $v$, and $\trop(W)$ is defined as the set $\{\trop(\phi)\:|\:\phi\in W\}\subseteq 2^E$.
\end{definition}

It seems commonly known among experts in matroid theory that $\trop (W)$ has a natural matroid structure by considering the elements with minimal support as the circuits. 
For the sake of completeness, we provide a proof by describing the circuits.

\begin{definition}\label{def_matroid_min_supp}
Let $\{0\}\neq W\subset K^E$ be a finite-dimensional $K$-vector subspace. We define $\mathcal{C}(W)\subset 2^E$ to be the sets in $\trop (W)\setminus\emptyset$ that are minimal with respect to set inclusion.
\end{definition}

\begin{example}
\label{ex:nice_example}
The set $\mathcal{C}(W)$ may consist of infinitely many elements. 
Consider for instance $W$ generated by $\varphi_1=1 + \sum_{i \ge 2} t^i$, $\varphi_2=\sum_{i \ge 1} i\,t^i$. 
Then, for every $n>1$, we obtain that $\phi_n:=n \cdot \varphi_1-\varphi_2 \in W$ has the support $\trop(\phi_n) = \mathbb{N} \setminus \{n\}$. 
Since there is no non-zero element in $W$ whose support is a subset of $\trop(\phi_n)$, $\trop(\varphi_1)$ or $\trop(\varphi_2)$, we obtain that
\[ \mathcal{C}(W)=\{ \mathbb{N} \setminus \{n\}\::\: n \in \mathbb{N} \} . \]
\end{example}

Note that, by definition, $\mathcal{C}(W)$ satisfies item (i) and (ii) from Definition~\ref{Def_CIM} by explicitly excluding the empty set and only considering minimal sets which cannot be subsets of other minimal sets. 
The next result shows that $\mathcal{C}(W)\neq\emptyset$.
\begin{lemma}
\label{lem:Cnonempty}
Let $W$ be as above. For every element $0\neq\varphi \in W$ there is $0\neq\psi \in W$ of minimal support such that $\trop(\psi) \subset \trop(\varphi)$.
\end{lemma}
\begin{proof}
Since $\{0\}\neq W$, we have $s=\text{dim}_K(W)>0$.
If $s=1$, then $W=K\cdot \varphi_1$ with $0\neq \varphi_1\in K^E$, and it is clear that $\mathcal{C}(W)=\{S_1=\trop(\varphi_1)\neq\emptyset\}$.

Suppose that $s>1$.
Consider $0\neq\varphi_1\in W$ and let $S_1=\mathcal T(\varphi_1)$. 
If $S_1$ is minimal, we are done. Otherwise, by definition of minimality, there should exist  $\emptyset\subsetneq S_2\subsetneq S_1$ corresponding to some $0\neq \varphi_2\in W$. Then $\{\varphi_1,\varphi_2\}\subset W$ is linearly independent, otherwise, an expression $\lambda_1\varphi_1+\lambda_2\varphi_2=0$ with $(\lambda_1,\lambda_2)\neq(0,0)$  would imply that $\lambda_1,\lambda_2\neq0$, and this would yield $S_1=S_2$.
Now repeat the process: if $S_2$ is minimal, we are done, otherwise, there exists $\emptyset\subsetneq S_3\subsetneq S_2$ corresponding to some $0\neq \varphi_3\in W$. Then the chain $\emptyset\subsetneq S_3\subsetneq S_2\subsetneq S_1$ implies that $\{\varphi_1,\varphi_2,\varphi_3\}\subset W$ is linearly independent, otherwise, an expression  $\lambda_1\varphi_1+\lambda_2\varphi_2+\lambda_3\varphi_3=0$ with $(\lambda_1,\lambda_2,\lambda_3)\neq(0,0,0)$ would  yield at least one equality on the chain $S_3\subsetneq S_2\subsetneq S_1$. This process eventually finishes since the dimension of $W$ is finite.
\end{proof}

Given $W$ as above, below we show that $\mathcal{C}(W)$ also satisfies (3) and (4) from Definition~\ref{Def_CIM}, so  $M(W)=(E,\mathcal{C}(W))$ is indeed a matroid.

\begin{remark}
Suppose that $W=K\cdot \varphi_1$ with $0\neq \varphi_1\in K^E$. Then $\mathcal{C}(W)=\{\trop(\varphi_1)\}$, and so $(E,\mathcal{C}(W))$ is a matroid.
\end{remark}

Take a basis $\{ \varphi_1,\ldots ,\varphi_s\} \subset W\subset K^E$ of $W$ as a $K$-vector space. Then $W=\{\lambda\cdot\underline{\varphi}:=\sum_{i=1}^s\lambda_i \varphi_i \::\: \lambda \in K^s\}$.
Each element of the basis has an expression in the standard basis of $K^E$ of the form $\varphi_i = (a_{ij})_{j\in E}$ with $a_{ij}\in K$. 
For each $j\in E$, set $u^{(j)}:= (a_{1j},\ldots, a_{sj})\in K^s$. With this notation, an index $j\in E$ lies in the support of $\lambda\cdot\underline{\varphi}\in W$ if and only if $\lambda\cdot u^{(j)}\neq 0$, that is:
\begin{equation}\label{SoporteLambdaVarphi}
    \trop (\lambda\cdot\underline{\varphi})= \{ j\in E \::\: \lambda\cdot u^{(j)}\neq 0\}.
\end{equation}
We keep this notation for the following lemmata.

\begin{lemma}\label{DimensionDeSubsepacioDeVectoresComplementarios}
Given $X\subset E$, there exists $0\neq \phi\in W$ with $\trop (\phi ) \subset X$ if and only if the $K$-linear subspace of $K^s$ generated by $\{ u^{(j)}\}_{j\notin X}$ is a proper subspace of $K^s$.
\end{lemma}
\begin{proof}
As a consequence of~\eqref{SoporteLambdaVarphi}, we have that $\trop (\lambda\cdot\underline{\varphi})\subset X$ if and only if $0\neq\lambda\in K^s$ is a solution of the system 
$$
\{u^{(i)}\cdot \lambda=0\}_{i\notin X}.
$$
This system has non zero solutions if and only if the $K$-linear subspace of $K^s$ generated by $\{ u^{(j)}\}_{j\notin X}$ is a proper subspace of $K^s$. 
Since $W=\{\lambda \cdot\underline{\varphi} \::\: \lambda \in K^s\}$, the statement follows, 
{as the dimension of the space of solutions of a linear system $AX=0$ in $s$ unknowns is $s$ minus the rank of the matrix $A$.}
\end{proof}
\begin{lemma}\label{Columnas de soporte de funcion}
Given $0 \ne \phi\in W$, let $L\subset K^s$ be the $K$-linear subspace generated by  $\{ u^{(i)}\}_{i\notin \trop (\phi )}$. Then, $L$ is a proper subspace of $K^s$ and $u^{(i)}\notin L$ for all $i\in \trop (\phi )$.
\end{lemma}
\begin{proof}
As a direct consequence of Lemma~\ref{DimensionDeSubsepacioDeVectoresComplementarios}, $L$ is a proper subspace of $K^s$. Now, $\phi$ is an element of $W$ if and only if it is of the form $\lambda\cdot\underline{\varphi}$ and an element $i\in E$ is in the support of $\lambda\cdot\underline{\varphi}$ if and only if $\lambda\cdot u^{(i)}\neq 0$. Since $L$ is generated by the set $\{ u^{(i)}\}_{i\notin \trop (\phi )}$, we have that $\lambda\cdot u^{(i)}=0$ for all elements in $L$.
\end{proof}

\begin{lemma}\label{condiciones para ser minimal}
Given $C\subset E$, let $L\subset K^s$ be the space spanned by $\{ u^{(i)}\}_{i\notin C}$. Then $C\in\mathcal{C}(W)$ if and only if $L$ is ($s-1$)-dimensional and $u^{(i)}\notin L$ for all $i\in C$.
\end{lemma}
\begin{proof}
Suppose that $C\in\mathcal{C}(W)$, then $C=\trop (\phi)$ for some $0 \ne \phi\in W$, and it follows from Lemma \ref{Columnas de soporte de funcion} that $L$ is a proper subspace of $K^s$ and $u^{(i)}\notin L$ for all $i\in C$. If the dimension of $L$ is $d< s-1$, we take $i_1,i_2,\ldots i_{s-d}$ in $C$ such that the subspace spanned by $L$ and $\{ u^{(i_k)}\}_{k=1,\ldots ,s-d}$ is $K^s$. The system $\{u^{(i)}\cdot \lambda=0\}_{i\notin C}\cup \{u^{(i_k)}\cdot \lambda=0\}_{k=2,\ldots ,s-d}\cup \{u^{(i_1)}\cdot \lambda=1\}$ has a unique solution $0\neq\lambda\in K^s$. For this solution, $\trop (\lambda\cdot\underline{\varphi})\subset C$. Since $\{i_k\}_{k=2,\ldots ,s-d}\subset C$ and $i_k\notin \trop (\lambda\cdot\underline{\varphi})$ for $k=2,\ldots ,s-d$, the set inclusion is strict (so $C$ is not minimal). So, we deduce that $L$ is $(s-1)$-dimensional.

Conversely, suppose that $L$ is ($s-1$)-dimensional. Then a solution of the system $\{u^{(i)}\cdot \lambda=0\}_{i\notin C}$ is unique up to scalar multiplication. If $0\neq\lambda\in K^s$ is such a solution, then the only elements of $W$ with support contained in $C$ are scalar multiples of $\phi := \lambda\cdot \underline{\varphi}$, so they all have the same support. Now $\trop (\phi)\subset C$ implies that the $K$-linear subspace generated by $\{ u^{(i)}\}_{i\notin \trop (\phi )}$ is a proper subspace of $K^s$ containing $L$, and since $L$ is ($s-1$)-dimensional, it is then the whole $L$.

Now, if $i\in C$ implies that $u^{(i)}\notin L$, then $u^{(i)}\in L$ if and only if $i\notin C$; indeed, the converse $i\notin C$ implies $u^{(i)}\in L$ is implied by the construction of $L$. Likewise, $u^{(i)}\in L$ if and only if $i\notin \trop (\phi )$ by Lemma~\ref{Columnas de soporte de funcion} and the previous paragraph. This yields $C=\trop (\phi )$ and thus  $C\in\mathcal{C}(W)$.
\end{proof}

\begin{lemma}\label{ExisteMinimalQueContieneElemento}
Given $0\neq \phi\in W$,
for every $z\in \trop (\phi )\subset E$ there exists a minimal element $C\in \trop (W)\setminus\emptyset$ such that $z\in C\subset \trop (\phi )$.
\end{lemma}
\begin{proof}
Let $L$ be the subspace generated by $\{ u^{(i)}\}_{i\notin \trop (\phi )}$.
By Lemma~\ref{DimensionDeSubsepacioDeVectoresComplementarios}, the subspace $L$ is of dimension $0<d< s$. Moreover, since $\phi$ is a combination of the $\varphi_i$'s, for all $i\in \trop (\phi )$ it holds that $u^{(i)}\notin L$. 
\begin{enumerate}
    \item 
    If $d=s-1$, then, by Lemma~\ref{condiciones para ser minimal}, $\trop (\phi )$ is minimal.
    \item
    Suppose that the dimension of $L$ is $d\leq s-2$. Take $i_2,\ldots i_{s-d}$ in $C$ such that the subspace generated by $L$, $u^{(z)}$ and $\{ u^{(i_k)}\}_{k=2,\ldots ,s-d}$ is $K^s$. The system $\{u^{(i)}\cdot \lambda=0\}_{i\notin C}\cup \{u^{(i_k)}\cdot \lambda=0\}_{k=2,\ldots ,s-d}\cup \{u^{(z)}\cdot \lambda=1\}$ has a unique solution $0\neq\lambda\in K^s$. For this solution, $z\in \trop (\lambda\cdot\underline{\varphi})\subset C\setminus \{i_k)\}_{k=2,\ldots ,s-d}$. Since the subspace generated by $L$ and $\{ u^{(i_k)}\}_{k=2,\ldots ,s-d}$ is of dimension $s-1$, again by Lemma~\ref{condiciones para ser minimal}, $\trop (\lambda\cdot\underline{\varphi})$ is minimal.\qedhere
\end{enumerate}
\end{proof}

\begin{lemma}
\label{EnEspacioFinitoLaXEsFinita}
Let $X\subset E$ and let $\{\phi_x\::\:x\in X\}\subset W$ be a set such that $\trop (\phi_x)\cap X=\{x\}$ for all $x\in X$. Then the vectors $\{ u^{(x)}\::\:x\in X\}$ are linearly independent.
\end{lemma}
\begin{proof}
For each $x\in X$, let $\lambda_x\in K^s$ be such that $\phi_x=\lambda_x\cdot \underline{\varphi}$ and let $L_x$ be the linear subspace generated by $\{ u^{(i)}\}_{i\notin\trop (\phi_x )}$. By Lemma~\ref{Columnas de soporte de funcion}, since $x\in\trop (\phi_x)$, then $u^{(x)}\notin L_x$. Now, $\trop (\phi_x)\cap X=\{x\}$ implies that $u^{(y)}\in \{ u^{(i)}\}_{i\notin\trop (\phi_x )}$ for $y\in X\setminus \{x\}$, then $u^{(x)}$ is not in the space generated by $\{ u^{(i)}\}_{i\in X\setminus \{ x\}}\subset L_x$.
\end{proof}

\begin{lemma}\label{lem:3}
Property~\ref{sumacircuitos} of Definition~\ref{Def_CIM} holds for $(E, \mathcal C(W))$.
\end{lemma}
\begin{proof}
Suppose that $X\subset C\in\mathcal{C}$ and $\{C_x\::\:x\in X\}$ is a family of elements of $\mathcal{C}$ such that $x\in C_y$ iff $x=y$ for all $x,y\in X$. Then, since $W$ is finitely generated, by Lemma~\ref{EnEspacioFinitoLaXEsFinita}, $X$ is a finite set.

Choose $\varphi= (b_j)_{j\in E}\in W$ such that $\trop (\varphi )=C$ and, for each $x\in X$, choose $\varphi_x =({a_x}_j)_{j\in E}\in W$ such that $\trop (\varphi_x )=C_x$.

Set $\phi:= \varphi - \sum_{x\in X} \frac{b_x}{{a_x}_x}\varphi_x$. We have that $C\setminus(\bigcup_xC_x)\subset \trop (\phi )\subset (C\cup \bigcup_xC_x)\setminus X$ and the result follows from Lemma~\ref{ExisteMinimalQueContieneElemento}.
\end{proof}

\begin{lemma}\label{lem:4}
Property~\ref{HayMaximalDeIndependientes} of Definition~\ref{Def_CIM} holds for $(E,\mathcal C(W))$.
\end{lemma}
\begin{proof}
Let $I\subset X\subset E$ with $I\in\mathcal{I}$. If $X\in \mathcal{I}$, then $X$ is the maximal element we are looking for. Otherwise, there exists $\varphi\in W$ such that $\trop (\varphi )\subset X$. Then, by lemma \ref{DimensionDeSubsepacioDeVectoresComplementarios}, the dimension of the subspace of $K^s$ generated by $\{u^{(j)}\::\: j\notin I\}$ is $s$ and the dimension of the subspace of $K^s$ generated by $\{u^{(j)}\::\: j\notin X\}$ is $r$ with $r<s$. Then, there exist $\{i_{r+1},\ldots, i_s\}\subset  X\setminus I$ such that $\{ u^{(i_{r+1})}, \ldots ,u^{(i_{s})}\}$ together with $\{u^{(j)}\::\: j\notin X\}$ generate $K^s$.

If $\bar{I}:= X\setminus \{ u^{(i_{r+1})}, \ldots ,u^{(i_{s})}\}$ then  $\bar{I}$ is a maximal element of 
$\{I'\in \mathcal{I}\::\:I\subset I'\subset X\}$.
\end{proof}
 
\begin{theorem}\label{Teorema: es un matroide}
Let $\{0\}\neq W\subset K^E$ be a finite-dimensional $K$-vector subspace. Then $M(W)=(E,\mathcal{C}(W))$ is a matroid.
Moreover, any element of $\trop (W)$ is a union of circuits.
\end{theorem}
\begin{proof}
Items (i) and (ii) are fulfilled as stated after Definition~\ref{def_matroid_min_supp}, and Lemmata~\ref{lem:3} and~\ref{lem:4} show the remaining items. The fact that any element of $\trop (W)$ is a union of circuits follows from Lemma~\ref{ExisteMinimalQueContieneElemento}.
\end{proof}

The above result is closely related to \cite[Theorem 4.12]{AB2015} where the language of thin sum matroids is used, for more details see \S\ref{sec:thin sums}.

Recall from Example \ref{ex:nice_example} that the set $\mathcal{C}(W)$ may consist of infinitely many elements.  Nevertheless, the matroid $(E,\mathcal{C}(W))$ has some finite structure in the following sense.

\begin{theorem}\label{thm-tame}
The matroid $M=(E,\mathcal{C}(W))$ is cofinitary. In particular, $M$ is tame.
\end{theorem}
\begin{proof}
By Lemma~\ref{DimensionDeSubsepacioDeVectoresComplementarios}, we have that $X\subset E$ is independent if and only if the $K$-linear subspace generated by $\{u^{(j)}\}_{j\notin X}$ is $K^s$, which means that there exist $\{j_1,\ldots,j_s\}\subset E\setminus X$ such that $\{u^{(j_1)},\ldots,u^{(j_s)}\}$ is a linearly independent set.
Then the bases of the matroid are exactly $\mathcal{B}_M=\{E\setminus\{j_1,\ldots,j_s\}\::\:Span\{u^{(j_1)},\ldots,u^{(j_s)}\}=K^s\}$, and $\mathcal{B}_M^*=\{\{j_1,\ldots,j_s\}\::\:Span\{u^{(j_1)},\ldots,u^{(j_s)}\}=K^s\}$.
Thus, the circuits of the dual matroid will have at most $s+1$ elements.
\end{proof}

\begin{remark}
\label{rem:rep}
The fact that $\mathcal{B}_M^*=\{\{j_1,\ldots,j_s\}\::\:Span\{u^{(j_1)},\ldots,u^{(j_s)}\}=K^s\}$ says that the dual $M^*$ of $M=(E,\mathcal{C}(W))$ is precisely the  matroid defined by the linear independence of the family of vectors $\{u^{(j)}\::\:j\in E\}$, see~\cite[Definition 2.6]{AB2015}.

A matroid defined by linear independence of a family of vectors is always finitary (even if the space spanned by the vectors is infinite-dimensional, see \cite[p. 2]{AB2015}, \cite[Section 2.6]{Infinte_matroids}). Thus, our original matroid $(E,\mathcal{C}(W))=M=(M^*)^*$ will in general not be representable.
\end{remark}

{\subsection{Relationship with thin sums matroids}\label{sec:thin sums}

Again, take a basis $\{ \varphi_1,\ldots ,\varphi_s\} \subset W\subset K^E$ of $W$ as a $K$-vector space of the form $\varphi_i = (a_{ij})_{j\in E}$.  In Remark \ref{rem:rep}, we showed that the dual of our matroid $M=M(W)=(E,\mathcal{C}(W))$ is the matroid $N=M^*$ represented by the family $\{u^{(j)}\::\:j\in E\}\subset K^s$ where $u^{(j)}:= (a_{1j},\ldots, a_{sj})$. 

Having this in mind, it is possible to use the theory of thin  families and their associated thin sums matroids to give a much shorter proof of the fact that $M$ is a matroid. The necessary concepts and constructions are rather technical, and the interested reader can find the details in \cite{AB2015}. But since the proof is rather short, we include it here as an alternative.

The dual $N^*=M$ arises as a thin sums matroid over a thin family for the field $K$ by \cite[Theorem 1.2]{AB2015}, and \cite[Theorem 3.3]{AB2015} gives an explicit way to compute $N^*$: let $C$ be the family of all linear dependencies of the family $\{u^{(j)}\::\:j\in E\}$, which in this case consists of all the finitely supported elements $c$ of $K^E$ which are orthogonal to the set $\{ \varphi_1,\ldots ,\varphi_s\}$. Then  $p:E\xrightarrow{}K^C$ defined by $p(e)(c):=c(e)$ is a thin family, and the corresponding thin sums matroid $M_{ts}(p)$ is $N^*$. 

 If $v\in K^E$ and $F\subset E$, we denote by  $v|_F$ the element of $K^E$ having coordinates $(v|_F)_e=v_e$ if $e\in F$, and  $(v|_F)_e=0$ otherwise. Now we have the following result.

\begin{lemma}
The set of vectors $v\in K^E$ which are orthogonal to all $c\in C$ is precisely  $W$.
\end{lemma}
\begin{proof}
Since $\{ \varphi_1,\ldots ,\varphi_s\}$ generate $W$ and are orthogonal to all the elements of $C$ by definition, then all the elements of $W$ are also orthogonal to all the elements of $C$.

Conversely, let $v\in K^E$ be an orthogonal vector to all the elements of $C$, and let $F\subset E$ be a finite set large enough that the set $\{ \varphi_1|_F,\ldots ,\varphi_s|_F\}$  is linearly independent. Then $v|_F$ is orthogonal to every element orthogonal to the set $\{ \varphi_1|_F,\ldots ,\varphi_s|_F\}$, and so we have an expression $v|_F=\sum_il_i\varphi_i|_F$ for uniquely determined $l_1,\ldots,l_s\in K$ since $\{ \varphi_1|_F,\ldots ,\varphi_s|_F\}$ is linearly independent. 

Given $e\in E\setminus F$, we know that $v|_{F\cup\{e\}}$ is orthogonal to every element orthogonal to the set $\{ \varphi_1|_{F\cup\{e\}},\ldots ,\varphi_s|_{F\cup\{e\}}\}$, and so we have an expression $v|_{F\cup\{e\}}=\sum_il_i'\varphi_i|_{F\cup\{e\}}$. Now we restrict this expression to $F\subset F\cup\{e\}$ to find that $l_i'=l_i$ for all $i$, and since this is true for every $e\in E\setminus F$, we deduce that $v=\sum_il_i\varphi_i\in W$.
\end{proof}

The previous Lemma shows that the set of thin dependencies $D_p$ of the thin family $p$ is precisely $W$, now \cite[Theorem 3.4]{AB2015} says that $M_{ts}(p)=M^*(\overline{p})$, where $\overline{p}:E\xrightarrow[]{}k^W$ is defined by $\overline{p}(e)(v)=v(e)$ for $e\in E$ and $v\in W$.


}

\section{Scrawls and cardinality} 
\label{Section:Scrawls}
In the previous section we have shown that the minimal elements of $\trop (W)$ satisfy the axioms of circuits. The scrawls of the matroid on $E$ given in terms of these circuits are unions of circuits. In this section we investigate the conditions on $E$ and $W$ upon which the collection of scrawls coincide with $\trop (W)$. 
For this purpose, let us denote by $\# C$ the cardinality of a set $C$ and by $S(C) = C \cup \{C\}$ the successor-set. 
Note that $\# S(C) = \# C+1$ for finite sets $C$ and $\# S(C) = \# C$ if $C$ is infinite.

\subsection{Matroids of scrawls}
Denote by $\lin (W)$ the set of $K$-linear subspaces $L\subset K^s$ with $L\neq K^s$ such that $L$ is generated by a set of the form $\{u^{(i)}\}_{i\in X}$ for some $X\subset E$. 
We will denote by $\Psi_W$ the  map given by
\begin{equation}\label{PsiW}
\begin{array}{cccc}
\Psi_W: & \lin (W) & \longrightarrow & 2^E\\
        & L &\mapsto & \{i\in E \::\: u^{(i)}\notin L\}.
\end{array}
\end{equation}

Notice that $\lin(W)$ is in fact independent of the choice of basis $\{\varphi_1,\dots,\varphi_s\}$. To see this consider another basis $\{\psi_1,\dots,\psi_s\}$ of $W$ and let $\psi_i=(b_{ij})_{j\in E}$. 
Denote for $j\in E$, $w^{(j)}:=(b_{1j},\dots,b_{sj})$.
Then there exists an invertible $s\times s$ matrix $\lambda=(\lambda_{ij})_{1\le i,j\le s}$ encoding the  change of basis: $\psi_i=\sum \lambda_{ij}\varphi_j$. In particular, for all $j\in E$ we have
\[
\lambda u^{(j)}=w^{(j)}.
\]
Now consider $X\subset E$ and $L=\langle u^{(j)}\::\: j\in X\rangle$. We have
\[
\langle u^{(i)}\rangle_{i\in X} = \langle \lambda^{-1} w^{(i)} \rangle_{i\in X} = \langle u^{(j)}\rangle_{j\in \mathcal T(\{\lambda^{-1}w^{(i)}\}_{i\in X})}.
\]
So the base change matrix $\lambda$ induces a natural bijection between the linear spaces generated by subsets of $\{u^{(i)}\}_{i\in E}$ and those of $\{w^{(i)}\}_{i\in E}$.

\begin{lemma}\label{La correspondencia entre circuitos y codimension uno}
The morphism $\Psi_W$ induces a one to one correspondence between the circuits of the matroid induced by $W$ and the spaces of codimension one in $\lin W$.
\end{lemma}
\begin{proof}
This is a direct consequence of Lemma~\ref{condiciones para ser minimal}.
\end{proof}

\begin{lemma}\label{Si es uno uno es cerrado por union}
The collection $\trop (W)\subset 2^E$ is closed under union if and only if $\Psi_W$ is a natural one to one correspondence between $\lin (W)$ and $\trop (W)$.
\end{lemma}
\begin{proof}
Lemma~\ref{Columnas de soporte de funcion} implies that $\trop (W) \subset \Psi_W (\lin (W)) $.

Now $\lin (W)$ is closed under intersection and $\Psi_W(L\cap L')=\Psi_W (L)\cup \Psi_W (L')$. Then, $\Psi_W (\lin (W))$ is closed under union. All elements of $\lin (W)$ can be written as intersection of elements of codimension one. Then, all the elements in  $\Psi_W (\lin (W))$ may be written in terms of images of elements of codimension one. By Lemma~\ref{La correspondencia entre circuitos y codimension uno}, images of elements of codimension one are circuits. By Theorem~\ref{Teorema: es un matroide}, all elements of $\trop (W)$ are union of circuits.
\end{proof}

\begin{lemma}\label{ExisteSubsepacioDeDimensionUnoMas}
Let $L$ be a $K$-linear subspace of $K^s$ with $\dim_K(L)=d\leq s-2$ and let $\{ u^{(i)}\}_{i\in X}\subset K^s$ with $X\subset E$ be 
such that $\{ u^{(i)}\}_{i\in X}\cap L=\emptyset$. 
If $\# X <\# S(K)$, then there exists a $K$-linear subspace $\overline{L}\supset L$ of $K^s$ of dimension $d+1$ with $\{ u^{(i)}\}_{i\in X}\cap \overline{L}=\emptyset$.
\end{lemma}
\begin{proof}
Let $A$ be the collection of $K$-linear subspaces of dimension $d+1$ of $K^s$ containing $L$. The collection $A$ is isomorphic to $\mathbb{P}_K^{s-d-1}$ where $\mathbb{P}_K$ denotes the projective space over $K$. 
Then the cardinality of $A$ is greater or equal to the cardinality of $S(K)$.
For every non-zero vector $u^{(i)} \notin L$, $i \in X$, since the dimension of $L$ is $d$, there exists exactly one $L_i \in A$ such that $u^{(i)}\in L_i$. Thus, if $\# X< \# S(K)$, there is $\overline{L} \in A \setminus \{L_i\}_{i \in X}$ and $\overline{L}$ does not contain any of the $u^{(i)} \notin {L}$.
\end{proof}

Let us note that for infinite $K$ the proof could be simplified by using $\#\mathbb{P}_K^{s-d-1}=\#\mathbb{P}_K=\#K$.

\begin{lemma}\label{CorrespondenciaSubespaciosSoporte}
Let $W$ be a subspace of $K^ E$ and let $\Psi_W$ be as in~\eqref{PsiW}. If $\# E < \# S(K)$, then $\Psi_W$ is a natural one to one correspondence between $\lin (W)$ and $\trop (W)$.
\end{lemma}
\begin{proof}
We start by showing that, for $L\in\lin (W)$, $\Psi_W (L)$ is in $\trop (W)$.
\begin{enumerate}
    \item If $L$ is of codimension one it is a consequence of Lemma~\ref{condiciones para ser minimal} (with no assumption about cardinality).
    \item Suppose that the dimension of $L$ is $d<s-1$. 
    Then the cardinality of $X:= \{ i\in E \::\: u^{(i)}\notin L\} $ satisfies $\#S^d(X)\leq \#E$ where $S^d$ denotes the $d$-th successor-set. 
    Applying $s-d-1$ times Lemma~\ref{ExisteSubsepacioDeDimensionUnoMas} , there exists a $K$-linear hyperplane $\overline{L}\supset L$ that does not contain any element of $\{u^{(i)}\}_{i\in X}$. 
    Take $\lambda\in K^s$ such that $\overline{L}= \{ v\in K^s \::\: \lambda\cdot v =0\}$. Since $L\subset \overline{L}$ and $\{u^{(i)}\}_{i\in E, u^{(i)}\notin L}\cap \overline{L} =\emptyset$ we have that $\lambda\cdot u^{(i)}=0$ for all $u^{(i)}\in L$ and $\lambda\cdot u^{(i)}\neq 0$ for all $u^{(i)}\notin L$. 
    Then, by~\eqref{SoporteLambdaVarphi}, we have that $\trop (\lambda\cdot \underline{\varphi}) =\Psi_W (L)$.
    \end{enumerate}
That the mapping is injective is straightforward and, to show surjectivity it is enough to see that $\trop (\phi)= \Psi_W (L)$ for any $\phi\in W$, where $L$ is the subspace generated by $\{ u^{(i)}\}_{i\in E\setminus\trop (\phi )}$.  
\end{proof}

\begin{theorem}\label{Teorema Son los Scrawls}
Let $W$ be a subspace of $K^E$. If $\# E < \# S(K)$, then $\trop (W)$ is a set of scrawls for $E$.
\end{theorem}
\begin{proof}
By Theorem~\ref{Teorema: es un matroide}, any element of $\trop (W)$ is a union of circuits. That any union of circuits is in $\trop (W)$ is a consequence of Lemma~\ref{CorrespondenciaSubespaciosSoporte} together with Lemma~\ref{Si es uno uno es cerrado por union}.
\end{proof}

The following example shows that the condition in Theorem~\ref{Teorema Son los Scrawls} about cardinality is optimal.

\begin{example}\label{Ejemplo de cuando la cardinalidad no es suficiente}
Let $E$ be a finite set and let $K$ be a finite field with $\# K< \# E$. 
Let $i_0\in E$ and set $a_{i_0}:=1$. Choose, for $i \in E$, the $a_i\in K$ such that $K=\{a_i\}_{i\in E \setminus \{i_0\}}$. 
Let $i_1\in E$  be such that $a_{i_1}=0$.
Set $\varphi_1:= \{ b_i\}_{i\in E}\in K^E$ where $b_i:=1$ for $i\neq i_0$ and $b_{i_0}=0$ and set $\varphi_{2}:= \{ a_i\}_{i\in E}\in K^E$ and let $W$ be the subspace of $K^E$ spanned by $\varphi_1$ and $\varphi_2$.

In this case, the image of the mapping~\eqref{PsiW} is not contained in $\trop (W)$ since $\Psi (\{ (0,\ldots ,0)\}) = E$ and $E$ is not the support of any element of $W$. 
In order to see this, take an element $\phi:= \lambda_1\varphi_1+\lambda_2\varphi_2\in W$. 
If $\lambda_1 \lambda_2=0$ then either $i_0\notin \trop (\phi)$ or $i_1\notin \trop (\phi)$. If $\lambda_1\lambda_2\neq 0$, let $i\in E$ be such that $a_i= -\frac{\lambda_1}{\lambda_2}$ (it exists because $K=\{ a_i\}_{i\in E \setminus \{i_0\}}$). Since $\trop (\phi)= \trop (\frac{1}{\lambda_2}\phi )= \trop (\frac{\lambda_1}{\lambda_2}\varphi_1+\varphi_2)= \trop ( \{\frac{\lambda_1}{\lambda_2}+ a_i\}_{i\in E} )\cup \{ i_0\}$ we have that $i\notin \trop (\phi)$.
\end{example}

Note that if $K$, $E$ and $W\subset K^ E$ satisfy the conditions of Theorem~\ref{Teorema Son los Scrawls}, then $\trop (W)$ is in particular a semigroup.
 
\begin{corollary}
\label{cor:new_result}
Let $K$ and $E$ be arbitrary. Let $s\in\mathbb{N}$, for each $j\in E$, set $u^{(j)}:= (a_{1j},\ldots, a_{sj})\in K^s$, such that $\{\varphi_i = (a_{ij})_{j\in E}\::\:i=1,\ldots,s\}$ is linearly independent. 

Let $M$ be the representable matroid induced by the family $\{u^{(j)}\::\:j\in E\}$. Then $M^*$ is the matroid of scrawls of $LinSpan\{\varphi_1,\ldots ,\varphi_s\}$ if $\#E< \# S(K)$.
\end{corollary}
\begin{proof}
Let $W=LinSpan\{\varphi_1,\ldots ,\varphi_s\}$. By Theorem \ref{thm-tame}, if $M$ is the representable matroid induced by the family $\{u^{(j)}\::\:j\in E\}$, then $M^*$ is its matroid of scrawls, and $M^*=\trop(W)$ if $\#E < \# S(K)$ by Theorem \ref{Teorema Son los Scrawls}.
\end{proof}

\begin{remark}
We have that Corollary~\ref{cor:new_result} says that even if a cofinitary matroid $M$ is representable by a family of column vectors, it does not follow automatically that its dual $M^*$ is the matroid of scrawls of the linear span of the row vectors.
\end{remark}

\begin{remark}
\label{rem:trop_lin}
Our concept of set of scrawls is stronger than that of a semigroup of $(2^E,\cup)$, since it is clear that a semigroup, being closed under unions, is spanned by its minimal elements, but it does not necessarily follow that this set of minimal elements satisfy the axioms of the circuits of a matroid. The concept of set of scrawls is also stronger than that of the circuits of a matroid on $2^E$, since it may happen that $a,b\in \mathcal{C}(W)$, but $a\cup b\notin \trop (W)$ (see Theorem~\ref{Teorema Son los Scrawls}).
\end{remark}

\section{Tropical linear spaces in the differential algebra setting}
\label{section:diff_tls}
We apply the previous theory to the case in which the set of formal solutions of a homogeneous system of linear differential equations is a finite dimensional vector space. 
The dimension of such solution spaces could be stated with the usage of D-modules~\cite{sattelberger2019d} or jet-spaces~\cite{kruglikov2006dimension}. 
When considering differential equations with polynomial coefficients instead of formal power series coefficients, one would speak of D-finite solutions~\cite{lipshitz1989d}. 
In this paper, however, we give a presentation by using the so-called differential type applicable to every system of algebraic (non-linear) differential polynomials as those considered in the \textsc{fttdag}~\cite{falkensteiner2023initials}. 

Throughout this section, we will consider $K$ to be a field of characteristic zero and $m,n \in \mathbb{N} \setminus \{0\}$.

\subsection{Preliminaries on differential algebra}
We denote by ${K}[\![T]\!]={K}[\![t_1,\ldots,t_m]\!]$ the ring of multivariate formal power series, by $D=\{\tfrac{\partial}{\partial t_1},\ldots, \tfrac{\partial}{\partial t_m}\}$ the set of standard partial derivatives, and we denote for $J\in\mathbb{N}^m$ the differential operator $\Theta(J)=\frac{\partial^{|J|}}{\partial t_{1}^{j_1} \cdots \partial t_m^{j_m}}$.
We denote by ${K}_{m,n}$ the polynomial ring ${K}[\![T]\!][x_{i,J}\::\: i=1,\ldots,n,\:J\in\mathbb{N}^m]$ where $x_{i}$ are differential indeterminates and $x_{i,J}:=\Theta(J)x_i$.
For $\Sigma\subset {K}_{m,n}$, we set $\Theta \Sigma=\{\Theta(J)f \::\: J \in \mathbb{N}^m, f \in \Sigma\}$.

\begin{definition}
\label{def:diff_id}
Let $\Sigma\subset {K}_{m,n}$. The differential ideal $[\Sigma]\subset {K}_{m,n}$ spanned by $\Sigma$ is the minimal ideal containing $\Sigma$ and being closed under taking derivatives.
\end{definition}

An element $P\in {K}_{m,n}$ is called a \textit{differential polynomial}, and the \textit{order} of $P$ is defined as the maximum of the $|J|=J_1+\cdots+J_m$ effectively appearing in $P$. The variables $x_{i,J}$, $i=1,\ldots,n$, $J\in\mathbb{N}^m$, in $K_{m,n}$ denote differential variables. We can define a map $P:{K}[\![T]\!]^n\xrightarrow[]{}{K}[\![T]\!]$ in which a monomial $E_M=\prod_{i,J}x_{i,J}^{m_{i,J}}$ sends the vector $\varphi=(\varphi_1,\ldots,\varphi_n)\in {K}[\![T]\!]^n$ to $E_M(\varphi)=\prod_{i,J}(\frac{\partial^{|J|}\varphi_i}{\partial t_{1}^{j_1} \cdots \partial t_m^{j_m}})^{m_{i,J}}$.

\begin{definition}
We say that $\varphi=(\varphi_1,\ldots,\varphi_n)\in {K}[\![T]\!]^n$ is a \textit{solution} of $P\in {K}_{m,n}$ if $P(\varphi)=0$. 
We denote by $\text{Sol}(P)$ the set of solutions of $P\in {K}_{m,n}$. Let $\Sigma\subset {K}_{m,n}$. The differential (algebraic) variety defined by $\Sigma$ is the set of common solutions $\text{Sol}(\Sigma)=\bigcap_{P\in \Sigma} \text{Sol}(P)\subset {K}[\![T]\!]^n$.
\end{definition}

Let $\Sigma \subset {K}_{m,n}$ be a system of differential equations such that the radical differential ideal generated by $\Sigma$ is prime. 
We need to require that $\Sigma$ is in a reduced form such that system such as $\{y, y'-1\}$ are further simplified before continuation. Formally this can for instance be achieved by imposing that $\Sigma$ is autoreduced~\cite{kolchin1973differential}. 
Let $L$ denote the set of leaders of $\Omega$. Then the transcendence degree of the general solution of $\Sigma$ is equal to the cardinality $d$ of the set $\Theta \{x_1,\ldots,x_n\} \setminus \Theta L$. 
We say that $\Sigma$ is of \textit{differential type zero} if and only if $d$ is finite. 
For the general notion of differential type and more details on the dimension of the solution set we refer to~\cite{kolchin1973differential, lange2014counting}.

\subsection{The linear case}
In this section we apply the results of the previous sections to the space of solutions $\text{Sol}(\Sigma)$ of a homogeneous linear system of differential equations $\Sigma\subset K_{m,n}$.
We start with some definitions.

\begin{definition}
\label{def:lin_eq}
An (algebraic) linear differential equation is a linear polynomial $P\in {K}_{m,n}$, i.e. $P=\sum_{i,J}\alpha_{i,J}x_{i,J}+\alpha$, with $\alpha_{i,J},\alpha\in {K}[\![T]\!]$. We say that $P$ is homogeneous if $\alpha=0$.
\end{definition}

If $P\in {K}_{m,n}$ is linear and homogeneous, then it is easy to see that $\text{Sol}(P)\subset {K}[\![T]\!]^n$ is a ${K}$-vector space. Thus if $\Sigma\subset {K}_{m,n}$ is a system of homogeneous linear differential equations, then $\text{Sol}(\Sigma)$ is also a ${K}$-vector space.

\begin{remark}\label{rem:differentialtypelinearcase}
Let us note that for homogeneous linear systems of differential equations $\Sigma$, the differential ideal generated by $\Sigma$ is always prime and every autoreduced set $\Omega$ of $\Sigma$ is homogeneous and linear as well. 
Moreover, the transcendence degree of the general solution $d$, if it is finite, is the dimension of $\text{Sol}(\Sigma)$~\cite[Chapter 3, Section 5]{kolchin1973differential}.
\end{remark}

With abuse of notation, we will denote by $\mathbb{N}^m$ the idempotent monoid $(\mathbb{N}^m,\cup,\emptyset)$, and we will denote by $\trop:{K}[\![T]\!]\xrightarrow[]{}2^{\mathbb{N}^m}$ the support map\footnote{The formal power series in the argument of $\trop$ can be identified as the list of coefficients such that this is consistent with Definition~\ref{def:supp_map}.}. If $n\geq 1$ and $X\subset {K}[\![T]\!]^n$, its set of supports (see \cite{falkensteiner2020fundamental}) is 
\begin{equation}
\label{eq:trop:eqdiff}
    \trop(X)=\{(\trop(w_1),\ldots,\trop(w_n))\in (2^{\mathbb{N}^m})^n\::\: (w_1,\ldots,w_n)\in X\}.
\end{equation}

The canonical order of the idempotent monoid $(2^{\mathbb{N}^m})^n$ defined by $(S_1,\ldots,S_n)\leq (T_1,\ldots,T_n)$ whenever $S_i\cup T_i=T_i$ for all $i=1,\ldots,n$ coincides with  the {\it product} or {\it cartesian order} $\leq_{prod}$ defined by $(S_1,\ldots,S_n)\leq_{prod} (T_1,\ldots,T_n)$ when $S_i\subseteq T_i$ for $i=1,\ldots,n$. Note that if $n\geq2$, then this  order differs from the inclusion order, thus, with the purpose of being able to apply the theory of the previous sections, first we need to perform the following transformation.


We only need the following extra procedures whenever $n\geq2$. Let us denote by $K[\![T^I\::\:I\in [n]\times\mathbb{N}^m]\!]$ the set of all maps $\varphi:[n]\times\mathbb{N}^m\xrightarrow[]{}K$ endowed with the sum given by point-wise addition of maps. We can represent the elements of this $K$-module as $\sum_{I\in [n]\times\mathbb{N}^m}a_It^I$.

We construct an isomorphism of $K$-modules $\Phi:({K}[\![t_1,\ldots,t_m]\!]^n,+)\to (K[\![T^I\::\:I\in [n]\times\mathbb{N}^m]\!],+)$ by sending the vector $\varphi=(\varphi_1,\ldots,\varphi_n)$ with $\varphi_i(I)=a_{i,I}$ to the map $\Phi(\varphi):[n]\times\mathbb{N}^m\xrightarrow[]{}K$ sending $(i,I)$ to $a_{i,I}$. The inverse map of $\Phi$ sends $\varphi\in K[\![T^I\::\:I\in [n]\times\mathbb{N}^m]\!]$ to $(\varphi_1,\ldots,\varphi_n)\in {K}[\![t_1,\ldots,t_m]\!]^n$, where $\varphi_i:=\sum_{I\in\mathbb{N}^m}\varphi(i,I)T^I$. 

This in turn induces an isomorphism of monoids $\phi:(2^{\mathbb{N}^m})^n\to 2^{[n]\times \mathbb{N}^{m}}$, where now $2^{[n]\times \mathbb{N}^{m}}$ is ordered by inclusion. The isomorphism sends the vector $(S_1,\ldots,S_n)$ to $\{1\}\times S_1\cup\cdots\cup \{n\}\times S_n$.


Then we have
\begin{equation}
    \label{eq:trop:eqdiff:chngd}
    \phi\circ\trop(X)=\{\{1\}\times\trop(w_1)\cup\cdots\cup\{n\}\times\trop(w_n)\in 2^{[n]\times \mathbb{N}^{m}}\::\: (w_1,\ldots,w_n)\in X\}.
\end{equation}

Thus, if $W\subset {K}[\![t_1,\ldots,t_m]\!]^n$ is a linear space, then $\Phi(W)\subset K[\![T^I\::\:I\in [n]\times\mathbb{N}^m]\!]$ is also a linear space which is isomorphic to $W$, and we have that $\phi\circ \trop(W)=\trop\circ\Phi(W) $, so we can apply the theory of the previous sections to the images under $\phi$ of set of supports \eqref{eq:trop:eqdiff} associated to finitely-dimensional vector spaces $W\subset {K}[\![t_1,\ldots,t_m]\!]^n$.

\begin{theorem}\label{Teorema_Matroide_Lin}
Let $\Sigma\subset {K}_{m,n}$ be a system of homogeneous linear differential equations of differential type zero and let $\trop(\text{Sol}(\Sigma))\subset (2^{\mathbb{N}^m})^n$ be the set of supports of $\text{Sol}(\Sigma)$. Then 
\begin{enumerate}
\item the minimal elements $\mathcal{C}(\Sol(\Sigma))$ of $\trop(\Sol(\Sigma))$ define the circuits of a matroid $M(\text{Sol}(\Sigma))$ on $[n]\times\mathbb{N}^{m}$;
    \item If $\# [n]\times\mathbb{N}^{m} < \# K$, then $\phi\circ\trop(\text{Sol}(\Sigma))\subset 2^{[n]\times \mathbb{N}^{m}}$ is a set of scrawls of the matroid $M(\text{Sol}(\Sigma))$.
\end{enumerate}
\end{theorem}

\begin{proof}
Since $\Sigma$ is of differential type zero, we have that $\{0\}\ne \text{Sol}(\Sigma)\subset {K}[\![T]\!]^n$ is a finite dimensional $K$-vector space.

We also have that the minimal elements of $\phi\circ\trop(\text{Sol}(\Sigma))\subset (2^{[n]\times\mathbb{N}^{m}},\subseteq)$ coincide with the minimal elements of $\trop(\text{Sol}(\Sigma))\subset ((2^{\mathbb{N}^m})^n,\leq_{prod})$, where $\leq_{prod}$ is the product order, since they  are isomorphic as posets. Thus $\phi(\mathcal{C}(\Sol(\Sigma)))=\mathcal{C}(\phi(\Sol(\Sigma)))$, so (i) follows from Theorem~\ref{Teorema: es un matroide}, and (ii) follows from Theorem~\ref{Teorema Son los Scrawls} after applying the inverse homomorphism $\phi^{-1}$ to the semigroup $\phi\circ\trop(\text{Sol}(\Sigma))\subset 2^{[n]\times\mathbb{N}^{m}}$.
\end{proof}

\begin{remark}
\label{rem:morphism:phi}
 If $n\geq2$, one must use the isomorphism $\phi:(2^{\mathbb{N}^m})^n\to 2^{[n]\times \mathbb{N}^{m}}$ in order to unveil the matroidal structures of the sets  $\text{Sol}(\Sigma)$. The isomorphism of posets yields that condition (iii) in Definition~\ref{Def_CIM} can be stated directly in terms of the poset $\trop(\text{Sol}(\Sigma))\subset ((2^{\mathbb{N}^m})^n,\leq_{prod})$. It should be interesting to see if the same can be done for condition (iv)  in Definition~\ref{Def_CIM}.
\end{remark}

\section{Applications to tropical differential algebraic geometry}
\label{Section:Connections}
We have that Theorem~\ref{Teorema_Matroide_Lin} from Section~\ref{section:diff_tls} is valid for the set of supports $\trop(\text{Sol}(\Sigma))\subset (2^{\mathbb{N}^m})^n$ of a system $\Sigma\subset {K}_{m,n}$ of homogeneous linear differential equations of differential type zero over an arbitrary field $K$. An important case occurs when $K$ \textit{satisfies the hypotheses of the \textsc{fttdag}}~\cite{boulier2021relationship}, namely when $K$ is an algebraically closed field of characteristic zero and has infinite transcendence degree over the field of definiton of $\Sigma$\footnote{Note that the latter is always fulfilled for uncountable fields.}.

\subsection{Tropical algebra preliminaries}
If $K$ satisfies the hypotheses of the \textsc{fttdag}, then we can express our results in a more algebraic form using the formalism of tropical algebra.
Recall that $2^{\mathbb{N}^{m}}=(2^{\mathbb{N}^{m}},\cup)$ is a semigroup.
The tropical counterparts of the underlying algebraic structures are as follows.
\begin{enumerate}
    \item Consider the Minkowski set sum $+:2^{\mathbb{N}^m}\times 2^{\mathbb{N}^m}\xrightarrow[]{}2^{\mathbb{N}^m}$;
    \item Define the {\it (tropical) differential operators} $D=\{\tfrac{\partial}{\partial t_i}:\mathbb{N}^m\xrightarrow[]{}\mathbb{N}^m\::\:i=1,\ldots,m\}$ by shifting the support accordingly as $$\tfrac{\partial}{\partial t_i}(S):=\{ (j_1, \ldots, j_{i-1}, j_i-1,j_{i+1}, \ldots, j_m) \::\: (j_1,\ldots,j_m) \in S, j_i > 0\}.$$
\end{enumerate}
Then the tuple $(2^{\mathbb{N}^m},\cup,+,D)$ is an (idempotent) differential semiring, see~\cite{falkensteiner2020fundamental}.

We give an alternative presentation of this algebraic structure. 
More details and the connection to the approach above is presented in~\cite{cotterill2020exploring}.

Let $\mathbb{B}=\{0<1\}$ be the Boolean semifield with the usual tropical addition $a+b:=min(a,b)$ and tropical multiplication $ab:=a+b$. If $\mathbb{B}[\![t_1,\ldots,t_m]\!]$ denotes the semiring of Boolean formal power series endowed with the standard operations of sum and product of Boolean power series, and $D=\{\tfrac{\partial}{\partial t_1},\ldots, \tfrac{\partial}{\partial t_m}\}$ denotes the set of standard partial derivations on $\mathbb{B}[\![t_1,\ldots,t_m]\!]$, then we have an isomorphism of (idempotent) differential semirings
\begin{equation}
\label{eq:base_set}
    (\mathbb{B}[\![t_1,\ldots,t_m]\!],+,\times,D)\cong(2^{\mathbb{N}^m},\cup,+,D).
\end{equation}
We denote by $\mathbb{B}_{m,n}$ the polynomial semiring $\mathbb{B}[\![T]\!][x_{i,J}\::\:i=1,\ldots,n,\:J\in\mathbb{N}^m]$. An element $P\in \mathbb{B}_{m,n}$ is called a differential polynomial, and the variables $x_{i,J}$, $i=1,\ldots,n$, $J\in\mathbb{N}^m$ in $\mathbb{B}_{m,n}$ denote differential variables.

We define a map $P:\mathbb{B}[\![T]\!]^n\xrightarrow[]{}\mathbb{B}[\![T]\!]$ in which a monomial $E_M=\prod_{i,J}x_{i,J}^{m_{i,J}}$ sends the vector $\varphi=(\varphi_1,\ldots,\varphi_n)\in \mathbb{B}[\![T]\!]^n$ to $E_M(\varphi)=\prod_{i,J}(\frac{\partial^{|J|}\varphi_i}{\partial t_{1}^{j_1} \cdots \partial t_m^{j_m}})^{m_{i,J}}$. 
The solutions of $P$ are defined in a tropical way as follows.

\begin{definition}
Given $A\in 2^{\mathbb{N}^m}$, its Newton polyhedron $\text{New}(A)$ is the convex hull of the set $\{I+J\::\: I\in A,J\in\mathbb{N}^m\}\subseteq\mathbb{R}_{\geq0}^m$.
We define the Newton polyhedron $\text{New}(\varphi)$ of $\varphi\in \mathbb{B}[\![t_1,\ldots,t_m]\!]$ by using the isomorphism from~\eqref{eq:base_set}.

The \emph{semiring of vertex polynomials} is defined as the quotient $V\mathbb{B}[T]:={\mathbb{B}[\![T]\!] }/{\text{New}}$, where $\text{New}\subset \mathbb{B}[\![T]\!] \times \mathbb{B}[\![T]\!]$ denotes the semiring congruence comprised of pairs of boolean power series with equal Newton polyhedra. We denote by $V:\mathbb{B}[\![T]\!]\xrightarrow[]{}V\mathbb{B}[T]$ the resulting quotient homomorphism of semirings.
\end{definition}

In the following we will denote the sum of equivalence classes in $V\mathbb{B}[T]$ by $``\oplus$''.
 
\begin{definition}
Given a sum
\begin{equation}
\label{eq:sum}
    s=a_1\oplus\cdots\oplus a_k
\end{equation}
in $V\mathbb{B}[T]$ involving $k \geq 2$ summands, let $s_{\:\widehat{i}}:=a_1\oplus\cdots\oplus\widehat{a_i}\oplus\cdots\oplus a_k$ denote the sum obtained by omitting the $i$-th summand, $i=1, \dots,k$.

The sum~\eqref{eq:sum} \emph{tropically vanishes} in $V\mathbb{B}[T]$ if $s= s_{\:\widehat{i}}\text{ for every }i=1,\ldots,k.$
\end{definition}

Given $P\in {K}_{m,n}$, the definition of a solution of $P$ can be given in a tropical way as follows.

\begin{definition}
\label{def_trop_sol}
We say that $\varphi=(\varphi_1,\ldots,\varphi_n)\in \mathbb{B}[\![T]\!]^n$ is a \textit{solution} of $\sum_Ma_ME_M=P\in \mathbb{B}_{m,n}$ if $$V(P(\varphi))=\bigoplus_MV(a_ME_M(\varphi))$$ vanishes tropically in $V\mathbb{B}[T]$. We denote by $\text{Sol}(P)$ the set of solutions of $P$. Recall that $E_M=\prod_{i,J}x_{i,J}^{m_{i,J}}$ represents a monomial and $a_M\in \mathbb{B}[\![T]\!]$ its corresponding coefficient. 
\end{definition}

Given $U\subset \mathbb{B}_{m,n}$ a system of tropical differential equations, we will denote by $\bigcap_{p\in U} \Sol(p)=\Sol(U)\subset \mathbb{B}[\![T]\!]^n$ its set of common solutions.

We will be mostly interested in the case where $P\in \mathbb{B}_{m,n}$ is linear as in Definition~\ref{def:lin_eq} where $\alpha_{i,J},\alpha\in \mathbb{B}[\![T]\!]$.
The tropical analogue of the fact that the set of solutions of a system of homogeneous linear differential equations is a vector space also holds for the case of homogeneous linear tropical differential equations in $\mathbb{B}_{m,n}$. Note that under the isomorphism~\eqref{eq:base_set}, the closure under taking unions translates to the closure under taking sums.
 
\begin{theorem}
\label{prop:B-Module}
Let $U\subset \mathbb{B}_{m,n}$ be a system of homogeneous linear tropical differential equations. Then $\Sol(U)\subset \mathbb{B}[\![T]\!]^n$ is a semigroup.
\end{theorem}
\begin{proof}
Since $\Sol(U)=\bigcap_{p\in U}\Sol(p)$, it suffices to show that $\Sol(p)$ is a semigroup for linear $p\in \mathbb{B}_{m,n}$. Let $\varphi,\psi\in \Sol(p)$ and set $\alpha=\varphi+\psi$, then $p(\alpha)=p(\varphi)+p(\psi)$, and it follows that $V(p(\alpha))=V(p(\varphi)+p(\psi))=V(p(\varphi))\oplus V(p(\psi))\subset V(p(\varphi))\cup V(p(\psi))$, which finishes the proof.
\end{proof}

So, if $U\subset \mathbb{B}_{m,n}$ is as in Theorem~\ref{prop:B-Module}, then $\Sol(U)\subset \mathbb{B}[\![T]\!]^n$ is a semigroup and thus, the union of tropical solutions are again tropical solutions. 
The structure of the semigroup $\Sol(U)$ was studied in~\cite{Trop_diff_eq} for the ordinary case ($m=1$) and $U$ finite. 
Following Remark \ref{rem:morphism:phi}, if $n\geq 2$, we can consider the image of $\Sol(U)$ under the map $\phi:\mathbb{B}[\![T]\!]^n\to \mathbb{B}[\![T^I\::\:I\in [n]\times\mathbb{N}^m]\!]$, which is also a semigroup, but it is not necessarily the set of scrawls of a matroid (see  Theorem~\ref{Teorema_Matroide_Lin}). 

{In Corollary~\ref{cor:real}, we give a necessary condition for a set $\phi(\Sol(U))$ to be the set of scrawls of a matroid. It would be interesting to find sufficient conditions under which a semigroup $\phi(\Sol(U))$ is the set of scrawls of a matroid.}

\subsection{Connections with the Fundamental Theorem}
In this section we analyze the special case when the coefficient field $K$ fulfills the hypotheses of the \textsc{fttdag}, this is,  ${K}$ is an uncountable algebraically closed field of characteristic zero.

To start with, by~\eqref{eq:base_set} we have an isomorphism of semirings $2^{\mathbb{N}^m}\cong\mathbb{B}[\![t_1,\ldots,t_m]\!]$, and we will denote by $\trop:{K}[\![t_1,\ldots,t_m]\!]\xrightarrow[]{}\mathbb{B}[\![t_1,\ldots,t_m]\!]$ the support map.  If $n\geq 1$ and $X\subset {K}[\![T]\!]^n$, its set of supports $\trop(X)\subset \mathbb{B}[\![t_1,\ldots,t_m]\!]^n$ is defined as in \eqref{eq:trop:eqdiff}.

Consider now a system $\Sigma\subset {K}_{m,n}$ of homogeneous linear differential equations of differential type zero over $K$. Then, by Theorem~\ref{Teorema_Matroide_Lin}, the set $\phi\circ\trop(\text{Sol}(\Sigma))$ is the set of scrawls of a matroid.

Now, the \textsc{fttdag} can be used to give a sufficient condition for {the solution set of tropical differential equations} to be the set of scrawls of a matroid, see~\ref{thm:fundamental}. This theorem gives an equality between the set $\trop(\Sol(\Sigma))$ and the set of formal Boolean power series solutions $\Sol(trop([\Sigma]))$ as in Definition~\ref{def_trop_sol} of some system $trop([\Sigma])\subset \mathbb{B}_{m,n}$.

\begin{definition}
\label{def:trop_pol}
Given $P=\sum_Ma_ME_M$ in $K_{m,n}$, we denote by $trop(P)$ the polynomial $trop(P)=\sum_M\trop(a_M)E_M$ in $\mathbb{B}_{m,n}$.  
\end{definition}

Recall that if $\Sigma\subset {K}_{m,n}$, then $[\Sigma]$ denotes the differential ideal spanned by it. The following result defines DA tropical varieties in three different ways, and also justifies the name.

\begin{theorem}\cite[Fundamental Theorem]{falkensteiner2023initials}\label{thm:fundamental}
Let $\Sigma\subset K_{m,n}$. 
Then the following three subsets of 
$\mathbb{B}[\![t_1,\ldots,t_m]\!]^n$ coincide.
\begin{enumerate}
    \item $X=\trop(\Sol(\Sigma))$;
    \item $X=\Sol(trop([\Sigma]))=\bigcap_{P\in [\Sigma]} \Sol(trop(P))$;
    \item $X=\{w\in \mathbb{B}[\![T]\!]^n\::\: in_w([\Sigma])\}$ contains no monomial.
\end{enumerate}
\end{theorem}

\begin{definition}
Any subset $X\subset \mathbb{B}[\![t_1,\ldots,t_m]\!]^n$ satisfying one of the characterizations of the above theorem is called a DA tropical variety.
\end{definition}

For a variety of examples and further discussion on the \textsc{fttdag}, see~\cite{boulier2021relationship}.

We now present the following result. 

\begin{corollary}
\label{cor:real}
Let $U\subset \mathbb{B}_{m,n}$ and $X=\bigcap_{p\in U} \Sol(p)$. If $U=trop([\Sigma])$, where $\Sigma\subset {K}_{m,n}$ is a system of homogeneous linear differential equations of differential type zero, then $\phi(X)$ is {the set of scrawls of $\phi(\mathcal{C}(\Sol(\Sigma)))$}.
\end{corollary}
\begin{proof}
Follows from the above result after applying the Fundamental Theorem~\ref{thm:fundamental}, since $\trop(\Sol(\Sigma))=\Sol(trop([\Sigma]))=\bigcap_{P\in [\Sigma]} \Sol(trop(P))=X$. Then we apply $\phi$ to both ends of the equality.
\end{proof}

We have shown that if $W=\Sol([\Sigma])$, where $\Sigma$ is as in Corollary~\ref{cor:real}, then $M(W)=(2^{[n]\times\mathbb{N}^{m}},\phi\circ\trop(W))$ is a matroid. Since a satisfactory theory of duality  exists for infinite matroids, and the dual  $M(W)^*$ of $M(W)$ is representable and finitary, it should be interesting to know if  its set of circuits $\mathcal{C}^*(W)$ relate to some notion of tropical basis for the ideal $[\Sigma]$.

\begin{example}
\label{ex:even_odd}
Let $\Sigma\subset K_{1,1}$ have the solutions $W$ generated by $\varphi_1=\sum_{i \ge 0} t^{2i}, \varphi_2=\sum_{i \ge 0} t^{2i+1}\in K[\![t]\!]$. Since $n=1$, there is no need to consider the map $\phi$, and it follows that
$$\trop(W) = \{ \trop(\varphi_1)= 2 \cdot \mathbb{N}, \trop(\varphi_2)= 2 \cdot \mathbb{N}+1, \emptyset, \mathbb{N} \}.$$
Thus, $\mathcal{C}= \{C_1=\trop(\varphi_1), C_2=\trop(\varphi_2)\}$. 
The bases of $\trop(W)$ are the maximal subsets of $\mathbb{N}$ that do not contain the set of even numbers nor the set of odd numbers. So they are complements of sets of the form $\{e,o\}$ where $e$ is an even number and $o$ an odd number. In particular, the bases of the dual are these pairs $\{e,o\}$.
The circuits of the dual (cocircuits) are now given as the pairs $\{e_1,e_2\}$, $\{o_1,o_2\}$ where $e_1,e_2$ are even and $o_1,o_2$ are odd numbers.
\end{example}

\section{Counterexample: the fundamental Theorem of tropical differential algebraic geometry over a countable field}
\label{sec:counterexp}
The article~\cite{Fundamental_theorem_TropDiffAlgGeom} by Aroca, Garay and Toghani proves a fundamental theorem for tropical differential algebraic geometry over uncountable fields. Using a non constructive result in~\cite{MR738249}, it was shown that there exists a system of algebraic partial differential equations over a countable field for which the fundamental theorem does not hold~\cite[Remark 7.3]{falkensteiner2020fundamental}. 
The question of whether the fundamental theorem holds for ordinary or even linear differential equations over countable fields remained open. 
As consequence of the results given in the previous sections, we give a negative answer to this question by constructing a counterexample as follows.

Fix a countable field $\Bbbk=\{a_0,a_1,a_2,\dots\}$, and consider the linear differential polynomial
\begin{equation}\label{eq:lin_diff_eq}
\Sigma:=\{y''+\gamma(t)y'+\beta(t)y\}\subset \Bbbk_{1,1}
\end{equation}
with $\gamma(t)=\sum_{i\ge 0}c_it^i$ and $\beta(t)=\sum_{i\ge 0}b_it^i$. In the following we want to build $b_i$ and $c_i$ such that Sol($\Sigma$) is generated by the two power series solutions (see Example~\ref{Ejemplo de cuando la cardinalidad no es suficiente})
\[
\varphi_1=1+\sum_{i\ge 2}t^i \quad \text{and} \quad \varphi_2=t+\sum_{i\ge 2}a_i t^i.
\]
Their supports are $\trop(\varphi_1)=\mathbb N\setminus \{1\}$ and $\trop(\varphi_2)=\mathbb N\setminus \{0\}$. 
For every $a_1,a_2\in \Bbbk$, also $a_1\phi_1-a_2\phi_2$ is a solution of $\Sigma$ and
\[
\trop(a_1\varphi_1-a_2\varphi_2)=\mathbb N\setminus \{j\}
\]
for some $j \in \mathbb{N}$ (or $\emptyset$ in the case of $a_1=a_2=0$). Note that there are no further solutions of $\Sigma$. The union $\trop(\varphi_1)\cup\trop(\varphi_2)=\mathbb{N}$ is a tropical solution (see Theorem~\ref{prop:B-Module}), but cannot be realized by any solution of~\eqref{eq:lin_diff_eq}.

Let us now construct the coefficients $b_i,c_i$ such that $\Sigma$ has a solution set generated by $\varphi_1,\varphi_2$. 
By plugging $\varphi_1$ into~\eqref{eq:lin_diff_eq}, we obtain
\begin{align*}
0 &= \sum_{i \ge 0}(i+2)(i+1)t^{i} + (\sum_{i\ge 0}c_it^i) \cdot (\sum_{i \ge 0}(i+1)t^{i}-1) + (\sum_{i\ge 0}b_it^i) \cdot (\sum_{i \ge 0}t^{i}-t) \\
&=\sum_{i \ge 0}((i+2)(i+1) + \sum_{j=0}^i (i+1-j)c_j - c_{i} + \sum_{j=0}^i b_j - b_{i+1} )t^{i}.
\end{align*}
Similarly, by plugging-in $\varphi_2$ into~\eqref{eq:lin_diff_eq} and setting $a_0=0,a_1=1$,
\begin{align*}
0 &= \sum_{i \ge 0}(i+2)(i+1)a_{i+2}t^{i} + (\sum_{i\ge 0}c_it^i) \cdot (\sum_{i \ge 0}(i+1)a_{i+1}t^{i}) + (\sum_{i\ge 0}b_it^i) \cdot (\sum_{i \ge 0}a_it^{i}) \\
&=\sum_{i \ge 0}((i+2)(i+1)a_{i+2} + c_i+ \sum_{j=0}^{i-1} (i+1-j)c_ja_{i+1-j} + \sum_{j=0}^{i-1} b_ja_{i-j})t^{i}.
\end{align*}
By coefficient comparison in both equations, seen as polynomials in $t$, we obtain the system of recurrence equations
\begin{align*}
b_{i+1} &= (i+2)(i+1) + \sum_{j=0}^i (i+1-j)c_j - c_{i} + \sum_{j=0}^i b_j, \\
c_{i+1} &= -(i+3)(i+2)a_{i+3} -\sum_{j=0}^i (i+2-j)c_ja_{i+2-j} - \sum_{j=0}^i b_ja_{i+1-j}
\end{align*}
for every $i \in \mathbb{N}$. By choosing $b_0, c_0$ as any element in $\Bbbk$, $\gamma(t), \beta(t)$ are uniquely determined from the given solutions $\varphi_1,\varphi_2$. Let us fix $b_0, c_0 \in \Bbbk$.

Then, since $\varphi_1,\varphi_2$ are linearly independent and $\Sigma$ consists of a single linear ordinary differential equation of order two, the solution set is indeed given exactly as the linear combinations of $\varphi_1$ and $\varphi_2$.

\begin{remark}
Notice that $\Bbbk$ can be chosen as the algebraic closure of the rational numbers. Then $\gamma(t)$ and $\beta (t)$ are (non-convergent) formal power series and $\Sigma$ is a system of linear differential equations involving formal power series coefficients. We thus have shown that for this frequently used case the fundamental theorem does not hold. For holononomic systems, i.e. when all coefficients of the elements in $\Sigma$ are polynomial, this question remains open.
\end{remark}

\begin{remark}

Within this paper, we have worked with homogeneous linear differential systems together with linear spaces. We expect that the results can be generalized to affine representable matroids~\cite[Section 1.5]{Oxley} and (non-homogeneous) linear differential systems of differential type zero. 
\end{remark}

\section*{Acknowledgments}
The author F.A. was supported by the PAPIIT project IN113323 dgapa UNAM. L.B. was supported by the PAPIIT projects IA100122 and IA100724 dgapa UNAM and CONAHCyT project CF-2023-G-106. S.F. is partially supported by the grant PID2020-113192GB-I00 (Mathematical Visualization: Foundations, Algorithms and Applications) from the Spanish MICINN and by the OeAD project FR 09/2022. C.G. wishes to thank David Fernández-Bretón for valuable conversations. We would like to thank the anonymous referee for having suggested the incorporation of Section 2.3 in this paper.

\footnotesize{
\bibliographystyle{alpha}
\newcommand{\etalchar}[1]{$^{#1}$}

}

\end{document}